\documentclass{cimart}
\usepackage{xcolor, amstext, amsfonts, amssymb, amsbsy, latexsym}
\usepackage{enumerate}
\usepackage{hyperref}
\usepackage[T1]{fontenc}
\usepackage{xy, hhline}
\usepackage{soul}
\usepackage[numbers,sort&compress]{natbib}
\newcommand{\St}{\mathop{\rm St}}

\numberwithin{equation}{section}

\newenvironment{eq}{\begin{equation}}{\end{equation}}

\newcommand{\Char}{\mathop{\rm char}}

\newcommand{\FF}{\mathbb{F}}

\newcommand{\RR}{\mathbb{R}}

\newcommand{\algA}{\mathcal{A}}
\newcommand{\algV}{\mathcal{V}}

\newcommand{\al}{\alpha}
\newcommand{\be}{\beta}
\newcommand{\ga}{\gamma}

\newcommand{\de}{\delta}

\newcommand{\LA}{\langle}
\newcommand{\RA}{\rangle}
\newcommand{\Id}[1]{{{\rm Id}({#1})}}
\newcommand{\ov}[1]{\overline{#1}}

\newcommand{\tr}{\mathop{\rm tr}}
\newcommand{\Aut}{\mathop{\rm Aut}}

\newcommand{\G}{{\rm G}_2}
\newcommand{\SL}{{\rm SL}}

\newcommand{\matr}[4]{\left(\begin{array}{cc}
#1 & #2 \\
#3 & #4 \\
\end{array}\right)}

\newcommand{\OO}{\mathbf{O}}

\newcommand{\uu}{\mathbf{u}}
\newcommand{\vv}{\mathbf{v}}
\newcommand{\cc}{\mathbf{c}}
\newcommand{\zero}{\mathbf{0}}

\newcommand{\new}[1]{#1}        
\newcommand{\del}[1]{}          

\title{
    On polynomial equations over split octonions
    }

\author{
    Artem Lopatin, Alexander Rybalov
    }

\authorinfo[Artem Lopatin]{State University of Campinas, Brazil}{dr.artem.lopatin@gmail.com}

\authorinfo[Alexander Rybalov]{Sobolev Institute of Mathematics, Russia}{alexander.rybalov@gmail.com}

\abstract{Working over the split octonions over an algebraically closed field, we solve all polynomial equations in which all the coefficients but the constant term are scalar.  As  a consequence,  we calculate the $n^{\rm -th}$ roots of an octonion.}



\begin{document}

\VOLUME{33}
\YEAR{2025}
\ISSUE{3}
\NUMBER{8}
\DOI{https://doi.org/10.46298/cm.14879}

\section{Introduction}

Assume that $\FF$ is a field of an arbitrary characteristic $p=\Char\FF\geq0$. All vector spaces and algebras are over $\FF$. 

The problem of solving polynomial equations was historically considered as one of key problems in mathematics, which influenced the creation of algebraic geometry and other branches of mathematics. Polynomial equations were considered not only over fields, but also over matrix algebras, algebras of quaternions, octonions, etc. Rodríguez-Ordó\~nez~\cite{Rodriguez-Ordonez_2007} proved that every polynomial equation over the algebra $\mathbf{A}_{\RR}$ of {\it Cayley numbers} (i.e., the {\it division algebra of real octonions}) of positive degree with the only term of the highest degree has a solution. An explicit algorithm for a solution of the quadratic equations $x^2 + bx +c = 0$ over $\mathbf{A}_{\RR}$ was obtained by Wag, Zhang and Zhang~\cite{Wang_Zhang_2014} together with criterions whether this equation has one, two or infinitely many solutions. 

In general, an {\it octonion algebra} $\mathbf{C}$ (or a {\it Cayley algebra}) over the field $\FF$ is a non-associative alternative unital algebra of dimension $8$, endowed with a non-singular quadratic multiplicative form $n: \mathbf{C} \to \FF$, which is called the {\it norm}. The norm $n$ is called {\it isotropic} if $n(a)=0$ for some non-zero $a\in \mathbf{C}$, otherwise the norm $n$ is {\it anisotropic}. In case $n$ is anisotropic, the octonion algebra $\mathbf{C}$ is a division algebra. In case $n$ is isotropic, there exists a unique octonion algebra $\OO_{\FF}$ over $\FF$ with isotropic norm (see Theorem 1.8.1 of~\cite{Springer_Veldkamp_book_2000}). This algebra is called the {\it split octonion algebra}. Note that, if $\FF$ is algebraically closed, then any octonion algebra is \new{isomorphic to} the split octonion algebra $\OO_{\FF}$ (for example, see Lemma 2.2 from~\cite{Lopatin_Zubkov_Eq_over_O}). Since Artin's theorem claims that in an alternative algebra every subalgebra generated by two elements is associative, then any octonion algebra is power-associative, i.e., the subalgebra generated by a single element is associative. Therefore, given $a\in \mathbf{C}$, we can write down $a^n$ without specifying the brackets in the product.

Flaut and Shpakivskyi~\cite{Flaut_Shpakivskyi_2015} considered the equation $x^n=a$ over real octonion division algebras. \new{For} an octonion division algebra $\mathbf{C}$ over an arbitrary field $\FF$, Chapman~\cite{Chapman_2020_JAA} presented a complete method \new{for finding} the solutions of the polynomial equation $a_n x^n + a_{n-1} x^{n-1} + \cdots + a_1 x + a_0 =0$ over $\mathbf{C}$. Moreover, Chapman and Vishkautsan~\cite{Chapman_Vishkautsan_2022}, working over a division algebra $\mathbf{C}$, determined the solutions of the polynomial equation  $(a_n c) x^n + (a_{n-1}c) x^{n-1} + \cdots + (a_1c) x + (a_0c) =0$  and discussed the solutions of the polynomial equation $(ca_n) x^n + (ca_{n-1}) x^{n-1} + \cdots + (ca_1) x + (ca_0) =0$. Chapman and  Levin~\cite{Chapman_Levin_2023} described a method for finding  so-called ``alternating roots'' of polynomials over an arbitrary division Cayley-Dickson algebra. Chapman and Vishkautsan~\cite{Chapman_Vishkautsan_2025} examined \new{the conditions under which},  for a root $a$ of a polynomial $f(x)$ over a general Cayley–Dickson algebra, there exists a factorization $f(x)=g(x)(x-a)$ for some polynomial $g(x)$.
The case of polynomial equations over an arbitrary algebra has recently been considered by Illmer and Netzer in~\cite{Illmer_Netzer_2024}, where  some conditions were determined  that imply the existence of a common solution of $n$ polynomial equations in $n$ variables,  with an application to polynomial equations over $\mathbf{A}_{\RR}$.

Assume that $\algA$ is a unital algebra over $\FF$. Consider a general polynomial equation $f(a_1,\ldots,a_m,x)=0$ over $\algA$, i.e., $f(a_1,\ldots,a_m,x)$ is an element of the absolutely free unital algebra $\FF\LA a_1,\ldots,a_m,x\RA$ with free generators $a_1,\ldots,a_m,x$. Here $x\in\algA$ is a variable and $a_1,\ldots,a_m\in\algA$ are coefficients. We \new{aim} to describe a method \new{for calculating} $x$ \new{when} $a_1,\ldots,a_m$ are given. Note that we can also consider $f(a_1,\ldots,a_m,x)$ as an element of the relatively free algebra $\FF_{\algA}\LA a_1,\ldots,a_m,x\RA:=\FF\LA a_1,\ldots,a_m,x\RA / \Id{\algA}$ for $\algA$, where $\Id{\algA}$ stands for the T-ideal of all polynomial identities for $\algA$. \new{By} knowing an $\FF$-basis for $\FF_{\algA}\LA a_1,\ldots,a_m,x\RA$, we can obtain a canonical form for $f(x)$. Therefore,   a problem of solving a polynomial equation over an algebra $\algA$ is tightly connected with the problem of explicit description of polynomial identities for $\algA$.  Although the theory of algebras with polynomial identities is a well-developed area of algebra with many deep results, there are still few results \new{offering} \new{an} explicit description of generators of T-ideals of polynomial identities for particular finite-dimensional algebras (see~\cite{  deMello_Souza_2023,  diniz2023two, diniz2024isomorphism,drensky2019varieties, Iritan_Alex_Artem_PIs_Novikov}


for recent results). Due to the difficulty of solving a general equation $f(a_1,\ldots,a_m,x)=0$ over $\algA$, it \new{may} be interesting to consider this equation over some vector subspace $\algV\subset \algA$ generating the algebra $\algA$, i.e., to assume that $a_1,\ldots,a_m,x\in\algV$. In this case, instead of polynomial identities, we should use so-called weak polynomial identities for the pair $(\algA,\algV)$ (see survey~\cite{Drensky_2021} and papers~\cite{Kislitsin_2022, Lopatin_Rodriguez_II, Lopatin_Rodriguez_III, Lopatin_Rodriguez_Tang_IV} for recent results on weak polynomial identities).

The split-octonions have numerous applications to physics. As an example, the Dirac equation, \new{which describes} the motion of a free spin $1/2$ particle, \new{such as} an electron or a proton, can be represented by the split-octonions (see~\cite{Gogberashvili_2006, Gogberashvili_2006_Dirac, Gogberashvili_Gurchumelia_2024}). There exist applications of split-octonions to electromagnetic theory (see~\cite{Chanyal_2017_CommTP, Chanyal_Bisht_Negi_2013, Chanyal_Bisht_Negi_2011}), geometrodynamics (see \cite{Chanyal_2015_RepMP}), unified quantum theories (see~\cite{Krasnov_2022, Bisht_Dangwal_Negi_2008, Castro_2007}), special relativity (see~\cite{Gogberashvili_Sakhelashvili_2015}).

In a recent paper, Lopatin and Zubkov~\cite{Lopatin_Zubkov_Eq_over_O} considered the linear equations $ax=c$, $(ax)b=c$, $a(bx)=c$ over the split octonion algebra $\OO$ in case $\FF$ is algebraically closed. Note that, over a division octonion algebra, these equations can easily be solved, and for non-zero $a,b$ there is a unique solution. \new{However,} in the case of an algebraically closed field, the situation is drastically different. \new{Specifically},  the set $X$ of all solutions of one of the above linear equations with non-zero $a,b$ is either empty, or contain\new{s} the only element, or \new{the} affine variety $X$ has dimension $r$, where
\begin{enumerate}
\item[$\bullet$]  $r=4$ in case we consider the equation $ax=c$; 

\item[$\bullet$]   $r\in\{4,5,7\}$ in case we consider the equation $(ax)b=c$; 

\item[$\bullet$]  $r\in\{4,6,8\}$ in case we consider the equation $a(bx)=c$.
\end{enumerate}

In Sections~\ref{section_O} and~\ref{section_main}  we assume that the field $\FF$ is algebraically closed. In Section~\ref{section_O} we explicitly define the octonion algebra $\OO$ and its group of automorphisms $\Aut(\OO)=\G$. In Section~\ref{section_main} we solve every equation 
\begin{eq}\label{eq_main}
\al_n x^n + \al_{n-1} x^{n-1} + \cdots + \al_1 x = c
\end{eq}%
\noindent{}with scalar $\al_1,\ldots,\al_n\in \FF$ and possibly non-scalar constant term $c\in\OO$ with respect to the variable $x\in\OO$ (see Theorem~\ref{theo_main}). The solution of equation~(\ref{eq_main}) is obtained modulo solution of polynomial equations over $\FF$. In Corollary~\ref{cor_number_solutions} we explicitly describe the number of solutions of equation~(\ref{eq_main}). In Corollary~\ref{cor_main} we apply the obtained general result to the $n^{\rm -th}$ roots of $c\in\OO$, i.e., to the solutions of the equation $x^n=c$.

\section{Octonions}\label{section_O}

In this section we assume that the field $\FF$ is algebraically closed.

\subsection{Split-octonions}\label{section_split}

The {\it split octonion algebra} $\OO=\OO(\FF)$, also known as the {\it split Cayley algebra}, is the vector space of all matrices

$$a=\matr{\al}{\uu}{\vv}{\be}\text{ with }\al,\be\in\FF \text{ and } \uu,\vv\in\FF^3,$$%
together with the multiplication given by the following formula:
$$a a'  =
\matr{\al\al'+ \uu\cdot \vv'}{\al \uu' + \be'\uu - \vv\times \vv'}{\al'\vv +\be\vv' + \uu\times \uu'}{\be\be' + \vv\cdot\uu'},\text{ where } a'=\matr{\al'}{\uu'}{\vv'}{\be'},$$%
$\uu\cdot \vv = u_1v_1 + u_2v_2 + u_3v_3$ and $\uu\times \vv = (u_2v_3-u_3v_2, u_3v_1-u_1v_3, u_1v_2 - u_2v_1)$. For short, we denote  $\cc_1=(1,0,0)$, $\cc_2=(0,1,0)$,  $\cc_3=(0,0,1)$, $\zero=(0,0,0)$ from $\FF^3$. Consider the following basis of $\OO$: $$e_1=\matr{1}{\zero}{\zero}{0},\; e_2=\matr{0}{\zero}{\zero}{1},\; \uu_i=\matr{0}{\cc_i}{\zero}{0},\;\vv_i=\matr{0}{\zero}{\cc_i}{0}$$
for $i=1,2,3$. The unity of $\OO$ is denoted by $1_{\OO}=e_1+e_2$.  We identify octonions
$$\al 1_{\OO},\;\matr{0}{\uu}{\zero}{0},\; \matr{0}{\zero}{\vv}{0}$$
with $\al\in\FF$, $\uu,\vv\in\FF^3$, respectively. Note that $\uu_i \uu_j= (-1)^{\epsilon_{ij}}\vv_k$ and  $\vv_i \vv_j= (-1)^{\epsilon_{ji}}\uu_k$, where $\{i,j,k\}=\{1,2,3\}$ and $\epsilon_{ij}$ is the parity of the permutation 
$\left(
\begin{array}{ccc}
\!1\! & \!2\! & \!3\! \\
\!k\! & \!i\! & \!j\! \\
\end{array}
\right)$.


The algebra $\OO$ is endowed with the linear involution
$$\ov{a}=\matr{\be}{-\uu}{-\vv}{\al},$$%
which satisfies the equality $\ov{aa'}=\ov{a'}\ov{a}$, a {\it norm} $n(a)=a\ov{a}=\al\be-\uu\cdot \vv$, and a non-degenerate symmetric bilinear {\it form} $q(a,a')=n(a+a')-n(a)-n(a')=\al\be' + \al'\be - \uu\cdot \vv' - \uu'\cdot \vv$. Define the linear function {\it trace} by $\tr(a)=a + \ov{a} = \al+\be$. The subspace of traceless octonions is denoted by $\OO_0 = \{a\in\OO\, | \,\tr(a)=0\}$ and the affine variety of octonions with zero norm is denoted by $\OO_{\#}=\{a\in\OO\,|\,n(a)=0\}$. Notice that
\begin{eq}\label{eq1}
\tr(aa')=\tr(a'a) \text{ and } n(aa')=n(a)n(a').
\end{eq}%
The next quadratic equation holds:
\begin{eq}\label{eq2}
a^2 - \tr(a) a + n(a) = 0.
\end{eq}%
The algebra $\OO$ is a simple {\it alternative} algebra, i.e., the following identities hold for $a,b\in\OO$:
\begin{eq}\label{eq4}
a(ab)=(aa)b,\;\; (ba)a=b(aa).
\end{eq}%
Moreover, 
\begin{eq}\label{eq4b}
\ov{a}(ab) = n(a) b, \;\; (ba)\ov{a}=n(a) b.  
\end{eq}%

The following remark is well-known and can easily be proven.

\begin{remark}\label{remark_inv}
Given an $a\in \OO$, one of the following cases holds:
\begin{enumerate}
\item[$\bullet$] if $n(a)\neq0$, then there exist unique $b,c\in\OO$ such that $ba=1_{\OO}$ and $ac=1_{\OO}$; moreover, in this case we have $b=c=\ov{a}/n(a)$ and we denote $a^{-1}:=b=c$.  

\item[$\bullet$] if $n(a)=0$, then $a$ does not have a left inverse as well as a right inverse.
\end{enumerate}
\end{remark}

Equalities~(\ref{eq4b}) imply that for $a\not\in\OO_{\#}$ we have
\begin{eq}\label{eq_inv}
a^{-1}(ab) = b, \;\; (ba)a^{-1}=b.    
\end{eq}

\subsection{The group \texorpdfstring{$\G$}{G}}\label{section_G2} 

The group $\Aut(\OO)$ of all automorphisms of the algebra $\OO$ is the exceptional simple group $\G=\G(\FF)$.  The group $\G$ contains a Zariski closed subgroup $\SL_3=\SL_3(\FF)$. Namely, every $g\in\SL_3$ defines the following automorphism of $\OO$:
$$a\to \matr{\al}{\uu g}{\vv g^{-T}}{\be},$$
where $g^{-T}$ stands for $(g^{-1})^T$ and $\uu,\vv\in\FF^3$ are considered as row vectors. For every $\uu,\vv\in \OO$ define $\de_1(\uu),\de_2(\vv)$ from $\Aut(\OO)$ as follows:
$$\de_1(\uu)(a')=\matr{\al' - \uu\cdot \vv'}{(\al'-\be' - \uu\cdot \vv')\uu + \uu'}{\vv' - \uu'\times \uu}{\be' + \uu\cdot\vv'},$$
$$\de_2(\vv)(a')=\matr{\al' + \uu'\cdot \vv}{\uu' + \vv'\times \vv}{(-\al'+\be' - \uu'\cdot \vv)\vv + \vv'}{\be' - \uu'\cdot\vv}.$$
The group $\G$ is generated by  $\SL_3$ and $\de_1(t\uu_i),\de_2(t\vv_i)$ for all $t\in\FF$ and $i=1,2,3$ (for example, see Section 3 of~\cite{zubkov2018}). By straightforward calculations, it is easy to see that
\begin{eq}\label{eq_h}
\hbar:\OO\to\OO, \text{ defined by }a \to \matr{\be}{-\vv}{-\uu}{\al},
\end{eq}%
belongs to $\G$.

The action of $\G$ on $\OO$ satisfies the next properties:
$$\ov{ga}=g\ov{a},\; \tr(ga)=\tr(a),\;n(ga)=n(a),\; q(ga,ga')=q(a,a').$$%
Thus, $\G$ acts also on $\OO_0$ and $\OO_{\#}$. The group $\G$  acts diagonally on the vector space $\OO^n=\OO\oplus \cdots\oplus\OO$ ($n$ times) by $g(a_1,\ldots,a_n)=(ga_1,\ldots,g a_n)$
for all $g\in \G$ and $a_1,\ldots,a_n\in\OO$.  

\new{We fix a binary relation $<$ on the field $\FF$ such that for each pair $\al,\be\in \FF$ with $\al\neq\be$ exactly one of $\al<\be$ or $\be< \al$ holds. Note that, we do not assume that $<$ is transitive, and we do not assume compatibility with the field operations.}

\begin{proposition}[Part 1 of Proposition 3.3 from~\cite{LZ_1}]\label{prop_O_canon}
 The following set is a minimal set of representatives of $\G$-orbits on $\OO$: 
\begin{enumerate}
\item[1.] $\al 1_{\OO}$,

\item[2.] $\matr{\al_1}{\zero}{\zero}{\al_2}$ with $\al_1<\al_2$,

\item[3.] $\matr{\al}{(1,0,0)}{\zero}{\al}$,
\end{enumerate}
where $\al,\al_1,\al_2\in\FF$. In other words, $\OO$ is the \new{disjoint} union of the following $\G$-orbits: 
$$\al 1_{\OO},\;\; O_2(\al_1,\al_2) := \G (\al_1 e_1 + \al_2 e_2),\;\;  O_3(\al) := \G (\al 1_{\OO} + \uu_1),$$
where $\al_1<\al_2$, $\al,\al_1,\al_2\in \FF$.
\end{proposition}

\begin{definition}\label{def_canon}{}
\begin{enumerate}
\item[$\bullet$] The elements from  Proposition~\ref{prop_O_canon} are called {\it canonical octonions}. 

\item[$\bullet$] Given $a\in\OO$, the diagonal elements of the canonical octonion from $\G a$ are called eigenvalues $(\al_1,\al_2)\in\FF^2$ for $a$, where $\al_1\leq\al_2$. Note that both eigenvalues are solutions of the equation $\al^2 - \tr(a) \al + n(a) = 0$. 
\end{enumerate}
\end{definition}

The following remark is a consequence of Part 2 of Proposition 3.3 from~\cite{LZ_1}.

\begin{remark}\label{remark_orbs} Assume $\al,\al_i,\be,\be_i,\ga_i\in\FF$ for $i=1,2$. 

\smallskip
\noindent{\bf 1.} Two octonions $\al_1 1_{\OO} + \be_1 \uu_1$ and  $\al_2 1_{\OO} + \be_2 \uu_1$ belong to the same $\G$-orbit on $\OO$ if and only if 
\begin{enumerate}
\item[$\bullet$] $\al_1=\al_2$,

\item[$\bullet$] either $\be_1=\be_2=0$, or $\be_1$, $\be_2$ are non-zero.
\end{enumerate}

\smallskip
\noindent{\bf 2.} If $\ga_1\neq \ga_2$, then $\al 1_{\OO} + \be \uu_1 \not\in O_2(\ga_1,\ga_2)$. 
\end{remark}

 
\section{Polynomial equations}\label{section_main}

In this section we assume that the field $\FF$ is algebraically closed. Given a commutative associative polynomial $f(\xi)=\al_n \xi^n + \cdots +\al_1 \xi + \al_0\in\FF[\xi]$, where $\al_0,\ldots,\al_n\in\FF$, we write $f'(\xi)$ for its derivative. For each $x\in\OO$ we naturally define the substitution \[f(x)=\al_n x^n + \cdots +\al_1 x + \al_0 1_{\OO}\in\OO.\] 

Assume that $f(\xi)\in\FF[\xi]$ is a non-zero polynomial without \new{constant term} and $\ga\in\FF$. The \new{multiplicity} of a root $\xi_1\in\FF$ for the equation $f(\xi)=\ga$, where $\xi\in\FF$ is a variable, is $k>0$ such that $f(\xi) \new{ - \ga} = (\xi-\xi_1)^k g(\xi)$ for some  $g(\xi)\in\FF[\xi]$ with $g(\xi_1)\neq0$. If $k=1$, then root $\xi_1$ is called {\it \new{simple}}. If $k\geq 2$, then \new{the} root $\xi_1$ is called {\it multiple}.  It is trivial that
\begin{enumerate}
\item[$\bullet$] $\xi_1$ is a simple root if and only if $f(\xi_1)=\ga$ and $f'(\xi_1)\neq0$;

\item[$\bullet$] $\xi_1$ is a multiple root if and only if $f(\xi_1)=\ga$ and $f'(\xi_1)=0$.
\end{enumerate}

\begin{lemma}\label{lemma_fx}
Assume that $f(\xi)\in\FF[\xi]$ and $\al,\be\in\FF$. Then \[f(\al 1_{\OO} + \be \uu_1) = f(\al) 1_{\OO} + f'(\al)\be\uu_1.\]
\end{lemma}
\begin{proof} For short, denote $a=\al 1_{\OO} + \be \uu_1$.

\medskip
\noindent{\bf 1.} Assume that $f(\xi)=\xi^n$ for some $n>0$. We prove by induction on $n$ that 
\begin{eq}\label{eq_claim1}
a^n=\al^n 1_{\OO} + n \al^{n-1} \be \uu_1.
\end{eq}

In case $n=1$ claim~(\ref{eq_claim1}) is trivial. 

Assume that claim~(\ref{eq_claim1}) holds for some $n>1$. 
Then \begin{multline*}a^{n+1} = (\al^n 1_{\OO} + n \al^{n-1}\be\uu_1) (\al 1_{\OO} + \be \uu_1) 
=\al^{n+1} 1_{\OO} + n \al^{n}\be\uu_1 + \al^n\be \uu_1 =\\
=\al^{n+1} 1_{\OO} + (n+1) \al^{n}\be\uu_1.
\end{multline*}
Therefore, claim~(\ref{eq_claim1}) holds for every $n>0$.

\medskip
\noindent{\bf 2.} Assume that $f(\xi)=\al_n \xi^n + \cdots +\al_1 \xi + \al_0\in\FF[\xi]$ for some $\al_0,\ldots,\al_n\in\FF$ and $n\geq0$. Note that in case $n=0$ we have $f(a)=\al_0 1_{\OO}$ and the claim of the lemma holds. 

For $n>0$ we apply part 1 to obtain that
$$f(a) = \sum_{i=1}^n \al_i a^i + \al_0 1_{\OO}  = \sum_{i=1}^n  \al_i (\al^i 1_{\OO} + i \al^{i-1}\be \uu_1) + \al_0 1_{\OO}, $$
and the required statement is proven.
\end{proof}

\begin{theorem}\label{theo_main} Assume that $f(\xi)\in\FF[\xi]$ is a non-zero polynomial without \new{constant} term and $c\in\OO$.  Acting by $\G$ on the equation $f(x)=c$, where $x\in\OO$ is a variable, we can assume that $c$ is a canonical  octonion from Proposition~\ref{prop_O_canon}.  Let $X\subset \OO$  be the set of all  solutions of the equation $f(x)=c$. Then 
\begin{enumerate}
\item[1.] in case $c=\ga 1_{\OO}$ for some $\ga\in\FF$, we have
$$\begin{array}{rl}
X= & \big\{\xi_1 1_{\OO} \,\big|\, \xi_1\in\FF \text{ satisfies } f(\xi_1)=\ga\big\} \bigcup \\
& \big\{ O_2(\xi_1,\xi_2) \,\big|\, \xi_1,\xi_2\in\FF \text{ satisfy } f(\xi_1)=f(\xi_2)=\ga \text{ and } \xi_1<\xi_2  \big \} \bigcup  \\
& \big\{O_3(\xi_1)\,\big|\,  \xi_1\in\FF \text{ is a multiple root for } f(\xi)=\ga \big\}; \\
\end{array}
$$

\item[2.] in case $c=\ga_1 e_1 +\ga_2 e_2$ for some $\ga_1,\ga_2\in\FF$ with $\ga_1<\ga_2$, we have
$$X=\big\{\xi_1 e_1 + \xi_2 e_2 \,\big|\, \xi_1,\xi_2\in\FF \text{ satisfy } f(\xi_1)=\ga_1 \text{ and }f(\xi_2)=\ga_2 \big \}; $$

\item[3.] in case $c=\ga 1_{\OO} +\uu_1$ for some $\ga\in\FF$, we have
$$X=\Bigg\{\xi_1 1_{\OO} + \frac{1}{f'(\xi_1)}\uu_1 \,\Bigg|\, \xi_1\in\FF \text{ is a simple root for } f(\xi)=\ga \Bigg \}.$$
\end{enumerate}
\end{theorem}
\begin{proof} Assume that there exists some $x\in X$. Consider  $g\in\G$ such that $gx$ is canonical. Note that 
\begin{eq}\label{eq_gxc}
f(gx)= g c.
\end{eq}

\medskip
\noindent{\bf 1.} Assume that $c=\ga 1_{\OO}$ for some $\ga\in\FF$. Then we can rewrite equality~(\ref{eq_gxc}) as
\begin{eq}\label{eq_gxc1}
f(gx)= \ga 1_{\OO}.
\end{eq}

\noindent{}One of the following possibilities holds:
\begin{enumerate}
\item[(a)] $gx=\xi_1 1_{\OO}$ for some $\xi_1\in \FF$. Equality~(\ref{eq_gxc1}) is equivalent to $f(\xi_1)=\ga$. Thus, $x=g^{-1} \xi_1 1_{\OO}=\xi_1 1_{\OO}$. 

\item[(b)] $gx=\xi_1 e_1 + \xi_2e_2$ for some $\xi_1,\xi_2\in \FF$ with $\xi_1<\xi_2$. Equality~(\ref{eq_gxc1}) is equivalent to $f(\xi_1)=f(\xi_2)=\ga$. Therefore, $x\in O_2(\xi_1,\xi_2)\subset X$.

\item[(c)] $gx=\xi_1 1_{\OO} +\uu_1$ for some $\xi_1\in \FF$. Lemma~\ref{lemma_fx}  implies that equality~(\ref{eq_gxc1}) is equivalent to $f(\xi_1)=\ga$ and $f'(\xi_1)=0$. Therefore,  $x\in O_3(\xi_1)\subset X$. 
\end{enumerate}

\medskip
\noindent{\bf 2.} Assume that $c=\ga_1 e_1 + \ga_2 e_2$ for some $\ga_1,\ga_2\in\FF$ with $\ga_1<\ga_2$. We have that one of the following possibilities holds:
\begin{enumerate}
\item[(a)] $gx=\xi_1 1_{\OO}$ for some $\xi_1\in \FF$. Equality~(\ref{eq_gxc}) implies that 
$$f(\xi_1) 1_{\OO} \in O_2(\ga_1,\ga_2);$$ 
a contradiction to Proposition~\ref{prop_O_canon}.

\item[(b)] $gx=\xi_1 e_1 + \xi_2e_2$ for some $\xi_1,\xi_2\in \FF$ with $\xi_1<\xi_2$. Equality~(\ref{eq_gxc}) is equivalent to 
$$f(\xi_1)e_1 + f(\xi_2) e_2 = g(\ga_1 e_1 + \ga_2 e_2).$$
Therefore,
$$f(\xi_1)e_1 + f(\xi_2) e_2 \in O_2(\ga_1,\ga_2).$$

Let $f(\xi_1)<f(\xi_2)$. Then Proposition~\ref{prop_O_canon} implies that $f(\xi_1)=\ga_1$ and $f(\xi_2)=\ga_2$. Hence, $g\in\St_{\G}(\ga_1 e_1 + \ga_2 e_2)$. Since $\St_{\G}(\ga_1 e_1 + \ga_2 e_2)=\SL_3$ by Lemma 2.1 of~\cite{LZ_1}, we have that  $x=g^{-1} (\xi_1 e_1 + \xi_2 e_2) = \xi_1 e_1 + \xi_2 e_2$.

In case $f(\xi_1)=f(\xi_2)$ we obtain a contradiction with Proposition~\ref{prop_O_canon}.

Let $f(\xi_1)>f(\xi_2)$. The equality $\hbar(f(\xi_1)e_1 + f(\xi_2) e_2) = \hbar g(\ga_1 e_1 + \ga_2 e_2)$ implies $f(\xi_2)e_1 + f(\xi_1) e_2 = \hbar g(\ga_1 e_1 + \ga_2 e_2)$. Thus, it follows from Proposition~\ref{prop_O_canon} that $f(\xi_1)=\ga_2$ and $f(\xi_2)=\ga_1$.  Hence, $\hbar g\in\St_{\G}(\ga_1 e_1 + \ga_2 e_2)$. Since $\St_{\G}(\ga_1 e_1 + \ga_2 e_2)=\SL_3$ by Lemma 2.1 of~\cite{LZ_1}, we have that  $x=(\hbar g)^{-1} \hbar (\xi_1 e_1 + \xi_2 e_2) = \xi_2 e_1 + \xi_1 e_2$.

\item[(c)] $gx=\xi_1 1_{\OO} +\uu_1$ for some $\xi_1\in \FF$. Lemma~\ref{lemma_fx} together with equality~(\ref{eq_gxc}) implies that 
$$f(\xi_1)1_{\OO} + f'(\xi_1) \uu_1 \in O_2(\ga_1,\ga_2);$$
a contradiction by part 2 of Remark~\ref{remark_orbs}.
\end{enumerate}

\medskip
\noindent{\bf 3.} Assume that $c=\ga 1_{\OO} + \uu_1$ for some $\ga\in\FF$.
We have that one of the following possibilities holds:
\begin{enumerate}
\item[(a)] $gx=\xi_1 1_{\OO}$ for some $\xi_1\in \FF$. Equality~(\ref{eq_gxc})  implies that
$$f(\xi_1) 1_{\OO}  \in O_3(\ga);$$
a contradiction to Proposition~\ref{prop_O_canon}.

\item[(b)] $gx=\xi_1 e_1 + \xi_2e_2$ for some $\xi_1,\xi_2\in \FF$ with $\xi_1<\xi_2$. Equality~(\ref{eq_gxc}) implies that
$$f(\xi_1)e_1 + f(\xi_2) e_2  \in O_3(\ga);$$
a contradiction to Proposition~\ref{prop_O_canon}.

\item[(c)] $gx=\xi_1 1_{\OO} +\uu_1$ for some $\xi_1\in \FF$. Lemma~\ref{lemma_fx}  implies that equality~(\ref{eq_gxc}) is equivalent to 
\begin{eq}\label{eq_gxc3c}
f(\xi_1) 1_{\OO} + f'(\xi_1)\uu_1 = g(\ga 1_{\OO} +\uu_1).
\end{eq}

\noindent{}It follows from part 1 of Remark~\ref{remark_orbs} that equality~(\ref{eq_gxc3c}) is equivalent to $f(\xi_1)=\ga$, $f'(\xi_1)\neq0$, and $g \uu_1 = f'(\xi_1)\uu_1$. Thus,
$$x=g^{-1}(\xi_1 1_{\OO} +\uu_1) = \xi_1 1_{\OO} + g^{-1}\uu_1 =
\xi_1 1_{\OO} + \frac{1}{f'(\xi_1)}\uu_1.$$
On the other hand, Lemma~\ref{lemma_fx} implies that for every $\xi_1\in\FF$ with $f(\xi_1)=\ga$ and $f'(\xi_1)\neq0$ we have  
$$f \Bigg(\xi_1 1_{\OO} + \frac{1}{f'(\xi_1)}\uu_1 \Bigg) = \ga 1_{\OO} + \uu_1.$$
\end{enumerate}

\end{proof}

The following corollaries are immediate consequences of Theorem~\ref{theo_main}.

\begin{corollary}\label{cor_number_solutions} 
Assume that $f(\xi)\in\FF[\xi]$ is a non-zero polynomial of degree $n> 1$ without constant term and $c\in\OO$.  Denote by $X\subset\OO$ the set of all solutions of the equation $f(x)=c$. Let $(\ga_1,\ga_2)\in\FF^2$ be \new{the} eigenvalues of $c$. Then 
\begin{enumerate}

\item[(a)]  $X$ is infinite if and only if $\ga_1=\ga_2$ and $c=\ga_1 1_{\OO}$;

\item[(b)] $X$ is empty if and only if $\ga_1=\ga_2$,  $c\neq\ga_1 1_{\OO}$, and the equation $f(\xi)=\ga_1$, where $\xi\in\FF$ is a variable, does not have any simple root;

\item[(c)] $|X|\leq n$ in case $\ga_1= \ga_2$ and  $c\neq\ga_1 1_{\OO}$;

\item[(d)] $1\leq |X|\leq n^2$ in case $\ga_1\neq \ga_2$.
\end{enumerate}
\end{corollary}

\begin{corollary}\label{cor_main} Assume that $n>1$ is an integer and $c\in\OO$. Acting by $\G$ on the equation $x^n=c$, where $x\in\OO$ is a variable, we can assume that $c$ is a canonical  octonion from Proposition~\ref{prop_O_canon}.  Let $X\subset\OO$  be the set of all  solutions of the equation $x^n=c$. Then 
\begin{enumerate}
\item[1.] in case $c=\ga 1_{\OO}$ for some $\ga\in\FF$, we have
$$
\begin{array}{rl}
X= & \big\{\xi_1 1_{\OO} \,\big|\, \xi_1\in\FF \text{ satisfies } \xi_1^n=\ga  \big \} \bigcup \\
& \big\{ O_2(\xi_1,\xi_2) \,\big|\, \xi_1,\xi_2\in\FF \text{ satisfy } \xi_1^n=\xi_2^n=\ga \text{ and } \xi_1<\xi_2  \big \} \bigcup Y, \text{ where }\\
Y= & \big\{O_3(\xi_1)\,\big|\,  \xi_1\in\FF \text{ satisfies } \xi_1^n=\ga\big\}  \;\text{ if } p | n \text{ or }\ga=0; \quad \text{and }
Y=\emptyset,\; \text{ otherwise}; \\
\end{array}
$$

\item[2.] in case $c=\ga_1 e_1 +\ga_2 e_2$ for some $\ga_1,\ga_2\in\FF$ with $\ga_1<\ga_2$, we have
$$X=\big\{\xi_1 e_1 + \xi_2 e_2 \,\big|\, \xi_1,\xi_2\in\FF \text{ satisfy } \xi_1^n=\ga_1 \text{ and }\xi_2^n=\ga_2 \big \}; $$

\item[3.] in case $c=\ga 1_{\OO} +\uu_1$ for some $\ga\in\FF$, we have
$$X=\Bigg\{\xi_1 1_{\OO} + \frac{\xi_1}{n \ga}\uu_1 \,\Bigg|\, \xi_1\in\FF \text{ satisfies } \xi_1^n=\ga \Bigg \}, \quad \text{ if } p \not{|} n \text{ and }\ga\neq 0;  $$
$$X=\emptyset, \quad \text{ if } p | n \text{ or }\ga= 0.
\qquad\qquad\qquad\qquad\qquad\qquad\qquad\qquad\qquad\quad\;\;$$
\end{enumerate}
\end{corollary}

\subsection*{Acknowledgements} We sincerely thank the four anonymous referees for their valuable and helpful comments. The first author was supported by FAPESP 2023/17918-2. The second author was supported by a state assignment for the Sobolev Institute of Mathematics, Siberian Branch, Russian Academy of Sciences, FWNF-2022-0003.




{\small\bibliography{cimart}}

\EditInfo{December 3, 2024}{March 5, 2025}{Ivan Kaygorodov, David Towers}
\end{document}